\documentclass{amsart}
\usepackage{hyperref}
\usepackage{amsmath,amsfonts,amsthm,amssymb}
\usepackage{graphicx}

\theoremstyle{plain}
\newtheorem*{theorem}{Theorem}
\newtheorem{lemma}{Lemma}

\newtheorem{hypothesis}{Hypothesis}

\theoremstyle{remark}

\DeclareMathOperator{\dv}{div}
\DeclareMathOperator{\dist}{dist}
\DeclareMathOperator{\e}{e}

\newcommand{\R}{\mathbb{R}}
\newcommand{\dd}{\,\mathrm{d}}

\makeatletter
\@addtoreset{hypothesis}{section}
\makeatother

\begin{document}

\title[Plateau angle conditions for the Allen--Cahn equation]{Plateau angle conditions for the vector-valued Allen--Cahn equation}
\author[N.~D.~Alikakos]{Nicholas D.~Alikakos}
\address{Department of Mathematics\\ University of Athens\\ Panepistemiopolis\\ 15784 Athens\\ Greece \and Institute for Applied and Computational Mathematics\\ Foundation of Research and Technology -- Hellas\\ 71110 Heraklion\\ Crete\\ Greece} 
\email{\href{mailto:nalikako@math.uoa.gr}{\texttt{nalikako@math.uoa.gr}}}
\author[P.~Antonopoulos]{Panagiotis Antonopoulos}
\address{Department of Mathematics\\ University of Athens\\ Panepistemiopolis\\ 15784 Athens\\ Greece} 
\email{\href{mailto:pananton@math.uoa.gr}{\texttt{pananton@math.uoa.gr}}} 
\author[A.~Damialis]{Apostolos Damialis}
\address{Department of Mathematics\\ University of Athens\\ Panepistemiopolis\\ 15784 Athens\\ Greece} 
\email{\href{mailto:damialis@math.uoa.gr}{\texttt{damialis@math.uoa.gr}}} 

\thanks{The research of the first author was partially supported by the Aristeia program of the Greek Secretariat for Research and Technology, co-financed by the European Union and the Greek state. The research of the third author has been co-financed by the European Union (European Social Fund) and Greek national funds through the Operational Program ``Education and Lifelong Learning'' of the National Strategic Reference Framework (NSRF)---Research Funding Program: Heraclitus II. Investing in knowledge society through the European Social Fund}

\begin{abstract}
Under proper hypotheses, we rigorously derive the Plateau angle conditions at triple junctions of diffused interfaces in three dimensions, starting from the vector-valued Allen--Cahn equation with a triple-well potential. Our derivation is based on an application of the divergence theorem using the divergence-free form of the equation via an associated stress tensor. 
\end{abstract}

\maketitle

\section{Introduction}
\label{section:introduction}
We consider the problem of determining contact angle conditions at triple junctions of diffused interfaces in three-dimensional space via the elliptic vector-valued Allen--Cahn equation
\begin{equation}\label{allen-cahn}
\Delta u - \nabla_{\! u} W (u) = 0
\end{equation}
for maps $u: \R^3 \to \R^3$ and a triple-well potential $W: \R^3 \to \R$. Equation \eqref{allen-cahn} is a vector analogue of the well-known scalar elliptic equation
\begin{equation}\label{scalar-allen-cahn}
\Delta u - W'(u) = 0
\end{equation}
for $u: \R^3 \to \R$ and a bistable potential $W:\R \to \R$ with two minima, which was introduced by Allen and Cahn \cite{allen-cahn} in the context of antiphase boundary motion. Here, $u$ is an order parameter that denotes the coexisting phases in the phenomenon of phase separation. We note that a vector-valued order parameter is necessary for the coexistence of three or more phases (see Bronsard and Reitich \cite{bronsard-reitich} and also Rubinstein, Sternberg, and Keller \cite{rubinstein-sternberg-keller}). Both equations are rescaled elliptic versions of corresponding evolution problems that involve a small parameter $\varepsilon$, which denotes the thickness of interfaces, and they are Euler--Lagrange equations of energy functionals, whose minimizers are related to minimal surfaces (see Modica and Mortola \cite{modica-mortola}, Modica \cite{modica}, and Baldo \cite{baldo}).

For the problem of contact angles in the case of soap films in three dimensions, the classical {\em Plateau angle conditions} state that
\begin{enumerate}
\item three soap films meet smoothly at equal angles of $120$ degrees along a curve,\label{law1}
\item four such curves meet smoothly at equal angles of about $109$ degrees at a point.\label{law2}
\end{enumerate}
The above laws hold in the isotropic case of soap films, which corresponds to systems of minimal surfaces (cf.\ Dierkes, Hildebrandt, and Sauvigny \cite[\S 4.15.7]{dierkes}). The angle of $120$ degrees in the first condition is the angle whose cosine is $-\frac{1}{2}$, which is exactly the angle in isotropic triple junctions in two dimensions, while the angle of about $109$ degrees in the second condition is the angle whose cosine is $-\frac{1}{3}$ (the so-called Maraldi angle). In the anisotropic case of mixtures of immiscible fluids, the angles above are not always equal and depend on the surface tension coefficients of each fluid, as in systems of constant mean curvature surfaces. In this case, the angles at the four triods determine the angles at the point singularity.

In \cite{bronsard-reitich}, for triple-well potentials, the authors linked at the level of formal asymptotics the diffused-interface problem with the associated sharp-interface problem and they established that to leading order the inner solution (that describes the solution near the triple junction) satisfies \eqref{allen-cahn}. That work is restricted on the plane. The rigorous study of \eqref{allen-cahn} for symmetric triple wells under the symmetries of the equilateral triangle was settled by Bronsard, Gui, and Schatzman \cite{bronsard-gui-schatzman} in two dimensions, and for symmetric quadruple-well potentials under the symmetries of the tetrahedron in three dimensions by Gui and Schatzman \cite{gui-schatzman}. Finally, Alikakos and Fusco \cite{alikakos-fusco-arma, alikakos-new-proof, fusco1} settled the general case of symmetric potentials under a general reflection group in arbitrary dimensions.

Concerning the minimal surface problem, Taylor \cite{taylor} showed that in three dimensions the only singular minimizing cones are a triod of planes meeting along a line at 120 degree angles, and a complex of six planes meeting with tetrahedral symmetry at a point, as in Plateau's laws. To our knowledge, results on minimal cones are not available in higher dimensions. Concerning the sharp-interface problem, the corresponding global geometric evolution problem in two dimensions for a single triple junction in a convex domain has been settled by Mantegazza, Novaga, and Tortorelli \cite{mantegazza-novaga-tortorelli} and Magni, Mantegazza, and Novaga \cite{magni-mantegazza-novaga}. Ilmanen, with co-authors, has announced very general results on the plane. The local existence for triple junctions on the plane was settled in Bronsard and Reitich \cite{bronsard-reitich}.

In this note we restrict ourselves to triple-well potentials defined over $\R^3$, with three global nondegenerate minima, and we assume the existence of an entire classical solution connecting the minima at infinity (see \cite{alikakos-fusco-arma, alikakos-new-proof, fusco1}). Such a solution partitions $\R^3$ in three regions that are separated by three interfaces which intersect along a line that we call the {\em spine} of the triod. In the symmetric case, this corresponds to one of the two singular minimizing cones for the associated Plateau problem.

Gui \cite{gui} considered the problem of determining the contact angles at a planar triple junction and under similar hypotheses to ours rigorously derived the so-called Young's law \cite{young}
\begin{equation}\label{youngs-law}
\frac{\sin \phi_1}{\sigma_{23}} = \frac{\sin \phi_2}{\sigma_{31}} = \frac{\sin \phi_3}{\sigma_{12}}
\end{equation}
for the angles $\phi_1$, $\phi_2$, $\phi_3$ between interfaces, where the $\sigma$'s are surface tension coefficients between neighboring phases. (See also the formal derivation in \cite{bronsard-reitich}.) We note that in three dimensions, the angles at a triod also obey Young's law, while the angles at a quadruple junction are a geometric consequence of the angle conditions at the four triods that form it (see Bronsard, Garcke, and Stoth \cite{bronsard-garcke-stoth} for the calculation), that is, given a quadruple-junction configuration, Plateau's second law follows from the first, which is the isotropic version of Young's law.

Our goal in the following is to extend Gui's work to three dimensions and rigorously derive Young's law for triods of interfaces as a property of solutions satisfying the next two hypotheses. (These will be justified and made precise in section \ref{section:statement}.) We note that the hypotheses below are satisfied by construction in Bronsard, Gui, and Schatzman \cite{bronsard-gui-schatzman} and Gui and Schatzman \cite{gui-schatzman}, and are theorems in \cite{alikakos-fusco-arma, alikakos-new-proof, fusco1,alikakos-fusco-hierarchy}.

\begin{hypothesis}
\label{h1}
Along rays emanating from any point of the spine, in the interior of regions, solutions converge exponentially to the corresponding minima of the potential.
\end{hypothesis}

\begin{hypothesis}
\label{h2}
Along lines parallel to interfaces and not parallel to the spine, in the interior of regions, solutions converge to {\em connections}, that is, to maps with argument the distance to the interface and with the property of connecting the minima of the potential at plus and minus infinity.
\end{hypothesis}

Our derivation makes use of the fact that \eqref{allen-cahn} is a divergence-free condition for a certain stress tensor, which appeared in Alikakos \cite{alikakos-basic-facts} in this context. For the planar analogue, the derivation of Bronsard and Reitich \cite{bronsard-reitich} is based on formal asymptotics, while the derivation of Gui \cite{gui} is closer to our spirit in terms of assumptions and rigor but uses Pohozaev-like identities instead of the stress tensor. (See \cite{alikakos-faliagas} for the connection between the two.) For the three-dimensional problem, related results appear in the theses \cite{antonopoulos} and \cite{damialis}. 

It should be noted that the problem in three dimensions is substantially more involved compared to the planar case, due to the fact that the singularity at a planar triple junction is a point that can be isolated, while the spine of a triod in three dimensions cannot, since it extends to infinity, and similar difficulties also arise in higher dimensions. We remark that our method applies easily to the two-dimensional case and that in higher dimensions the application is significantly harder and will appear elsewhere \cite{antonopoulos-damialis}.

The rest of the present note consists of two sections. In section \ref{section:statement} we set up the problem, stating all necessary assumptions, and formulate it in divergence-free form via the stress tensor. In section \ref{section:derivation} we apply the divergence theorem for the stress tensor on a ball, which is then blown up after a proper slicing. This slicing involves a surgery around the singularity which appears at the intersection of the surface of integration with the spine and a breaking up of the remaining part in a way that utilizes the hypotheses on the solutions at infinity, that is, Hypothesis \ref{h1} far from interfaces and Hypothesis \ref{h2} at fixed distances from them. This yields Young's law in the form of a balance of forces relation for the conormals of the three interfaces, which is equivalent to \eqref{youngs-law}.

\section{Statement of the problem and preliminaries}
\label{section:statement}

We start with \eqref{allen-cahn}
\[ \Delta u - \nabla_{\! u} W(u) = 0 \]
for $u: \R^3 \to \R^3$ and $W: \R^3 \to \R$, where $\nabla_{\! u} W (u) = (\partial W / \partial u_1, \partial W / \partial u_2, \partial W / \partial u_3)$. The potential $W$ is taken to be of class $C^2$, nonnegative, and with three nondegenerate global minima at points $a_1$, $a_2$, $a_3$, that is, $W(a_1)=W(a_2)=W(a_3)=0$, with $W(u)>0$ otherwise. Moreover, we ask that $W$ satisfies a mild coercivity assumption, that is, 
\[ \liminf_{\lvert u \rvert \to +\infty} W(u) > 0.\] 
Finally, note that we do not make any symmetry assumptions on $W$.

As explained in section \ref{section:introduction}, we consider solutions that partition the domain space in three regions via a triod of diffused flat interfaces. We distinguish three regions $C_i$ in $\R^3$ for $i=1, 2, 3$, such that the region $C_i$ contains the minimum $a_i$ with $\Gamma_{ij}$ being the interface that separates $C_i$ and $C_j$ (with $\Gamma_{ij} \equiv \Gamma_{ji}$). For each region $C_i$ we have that if $x \in C_i$, then $\lambda x \in C_i$ for $\lambda > 0$ (cone property).

We choose coordinates as follows. We take the origin on the spine, which we identify with the $x_3$ axis, and we further identify interface $\Gamma_{12}$ with the half plane $x_1 = 0$, $x_2 \geq 0$, $x_3 \in \R$, such that $x_1$ is the distance to $\Gamma_{12}$. We also recall here the spherical coordinates in three dimensions, that is,
\[ x_1 = r \cos \theta_1 \sin \theta_2,\quad x_2 = r \sin \theta_1 \sin \theta_2, \quad x_3 = r \cos \theta_2, \]
for an azimuthal angle $\theta_1 \in [0,2\pi)$, for a polar angle $\theta_2 \in [0,\pi]$, and for $r \geq 0$. In terms of the azimuthal angle $\theta_1$, the interface $\Gamma_{12}$ lies at $\theta_1 = \frac{\pi}{2}$.

The uniformly bounded entire solutions we consider satisfy 
\begin{equation}
\label{u-bound}
\lvert u (x) \rvert < C,
\end{equation}
globally in $\R^3$. Using this bound and linear elliptic theory (cf.\ the proof of Lemma \ref{regularity-lemma}), we also have the uniform bound
\begin{equation}
\label{grad-u-bound}
\lvert \nabla u (x) \rvert < C,
\end{equation}
again globally in $\R^3$. 

For such solutions we have two hypotheses. The first one concerns the fact that solutions converge exponentially to the corresponding equilibrium in the interior of each region. This has been verified under assumptions of symmetry on the potential by several authors (see \cite{bronsard-gui-schatzman, gui-schatzman,alikakos-fusco-arma,alikakos-new-proof,fusco1}) and we postulate that it holds for general potentials. 

\begin{hypothesis}[Exponential estimate]
\label{hypothesis1}
In the interior of the region $C_i$ there holds that 
\begin{equation}
\lvert u(x) - a_i \rvert \lesssim \e^{- \dist(x,\partial C_i)},
\end{equation}
where $\partial C_i =  \cup_{i \neq j} \Gamma_{ij}$.
\end{hypothesis}

(We use the notation $X \lesssim Y$ for the estimate $X \leq CY$, where $C$ is an absolute constant.)

The second hypothesis is that along directions parallel to interfaces solutions converge to ``one-dimensional'' heteroclinic connections $U_{ij} : \R \to \R^3$, which connect the equilibria $a_i$, $a_j$ at infinity, in the sense that
\[
\lim_{\eta \to -\infty} U_{ij}(\eta) = a_i \quad \text{and} \quad \lim_{\eta \to +\infty} U_{ij}(\eta) = a_j,
\]
where $\eta$ is the distance to the interface $\Gamma_{ij}$. 

\begin{hypothesis}[Connection hypothesis]
\label{hypothesis2}
Along directions parallel to an interface $\Gamma_{ij}$ solutions converge pointwise to a one-dimensional connection $U_{ij}(\eta)$ with argument the distance to the interface, that is, 
\begin{equation}
\lim_{\lvert x \rvert \to + \infty} u(x) = U_{ij} (\eta), \text{ for fixed } \eta := \dist(x,\Gamma_{ij}).
\end{equation}
\end{hypothesis}

These limiting functions are solutions to the associated Hamiltonian ODE system
\[
\ddot{U}_{ij} - \nabla W(U_{ij}) = 0
\]
with the property of connecting the minima of $W$ at infinity. (We refer to \cite{sternberg, alikakos-fusco-indiana, alikakos-betelu-chen} for further information on the connection problem.) We define the action of such a connection to be the nonnegative quantity
\begin{equation}
\sigma_{ij} = \sigma(U_{ij}) := \int_{-\infty}^{+\infty} \left( \frac{1}{2} \lvert \dot{U}_{ij} \rvert^2 + W(U_{ij}) \right) \dd \eta,
\end{equation}
and also note that connections satisfy the equipartition relation
\begin{equation}\label{equipartition}
\frac{1}{2} \lvert \dot{U}_{ij} \rvert^2 = W(U_{ij}).
\end{equation}

In support of the hypotheses above, we note that for potentials $W$ symmetric with respect to the group of symmetries of the equilateral triangle in $\R^3$, where $W$ is $C^2$ and with three global minima placed inside the partitioning regions, and under the coercivity assumption $\liminf_{\lvert u \rvert \to +\infty} W(u) > 0$, Hypothesis \ref{hypothesis1} is a theorem (cf.\ \cite{fusco1}). Moreover, Hypothesis \ref{hypothesis2} is also a theorem for symmetric potentials (cf.\ \cite{alikakos-fusco-hierarchy}), under the additional assumption of uniqueness and hyperbolicity of connections, which is generic by the appendix in \cite{gui-schatzman}. Finally, for nonsymmetric potentials $W$, Fusco \cite{fusco2} has established Hypothesis \ref{hypothesis1} for local minimizers of 
\[ \int \frac{1}{2} |\nabla u|^2 + W(u). \]
We note that the entire solution constructed by Alikakos and Fusco \cite{alikakos-fusco-arma, alikakos-new-proof, fusco1} is a local minimizer (cf.\ \cite{alikakos-fusco-hierarchy}). We expect that along the lines of \cite{fusco2} one should also be able to verify Hypothesis \ref{hypothesis2} for local minimizers. In conclusion, the major underlying hypothesis in our work is the existence of an entire solution connecting in a certain sense the minima for nonsymmetric potentials (cf.\ S\'aez Trumper\cite{saez}).

We will now reformulate \eqref{allen-cahn} via its associated stress tensor. This tensor was introduced in Alikakos \cite{alikakos-basic-facts} in the context of \eqref{allen-cahn}, where further applications are also presented. However, it is a well-known object in the physics literature (for instance, in the Landau and Lifshitz series; see \cite{alikakos-faliagas} for more information). We define the stress tensor $T$ as
\begin{equation}
\label{tensor-components}
T_{ij} (u) = u_{,i}\cdot u_{,j} - \delta_{ij} \left(\frac{1}{2} \lvert \nabla u \rvert^2 + W(u) \right)
\end{equation}
for maps $u: \R^n \to \R^m$, where $u_{,i} = \partial u / \partial x_i$ and where the dot denotes the Eu\-clid\-e\-an inner product in $\R^m$. In three dimensions (that is, for $n=3$) it is a $3 \times 3$ symmetric matrix
\[
T(u) = \frac{1}{2} 
\left( \begin{array}{c}
\lvert u_{,1} \rvert^2 - \lvert u_{,2} \rvert^2 - \lvert u_{,3} \rvert^2  - 2 W(u)\qquad 2 u_{,1} \!\cdot u_{,2}\qquad 2 u_{,1} \!\cdot u_{,3} \\
2 u_{,2} \!\cdot u_{,1}\qquad \lvert u_{,2} \rvert^2 - \lvert u_{,1} \rvert^2 - \lvert u_{,3} \rvert^2 - 2 W(u)\qquad 2 u_{,2} \!\cdot u_{,3} \\
2 u_{,3} \!\cdot u_{,1}\qquad 2 u_{,3} \!\cdot u_{,2}\qquad \lvert u_{,3} \rvert^2 - \lvert u_{,1} \rvert^2 - \lvert u_{,2} \rvert^2 - 2 W(u) \end{array} \right)
\]
with the property
\begin{equation}
\label{div-free}
\dv T = (\nabla u)^\top (\Delta u - \nabla_{\! u} W(u)),
\end{equation}
that is, $T$ is divergence-free when applied to solutions of \eqref{allen-cahn}. 

We also note that $T$ is invariant under rotations of the coordinate system, that is, it transforms as a tensorial quantity. To see this, consider an orthogonal transformation $Q$ and a new coordinate system $x' = Q x$. Letting $u'$ be the map acting on the new coordinates with $u'(x') = u(x)$, the chain rule gives that its gradient is transformed via $\nabla' u' = Q \nabla u$, where the prime denotes that the derivatives are taken with respect to the new coordinate system. Then, for the transformed tensor $T^{\prime}$, which is given by the similarity transformation
\[ T' = Q T Q^\top, \]
due to the form of the components in \eqref{tensor-components} and the continuity of $W$ there holds
\[ T^{\prime}_{ij} (u^\prime) = u^{\prime}_{,i} \cdot u^{\prime}_{,j} - \delta_{ij} \left(\frac{1}{2} \lvert \nabla^{\prime} u^{\prime} \rvert^2 + W(u^\prime) \right), \]
where again the prime denotes that the tensor is calculated in the new coordinate system. That is, the transformed tensor has exactly the same expression as the original one, except for the fact that it acts in the transformed coordinates.

Finally, we prove two lemmas that will be used in the following. The first is a consequence of Hypothesis \ref{hypothesis1} and linear elliptic theory, while the second follows from Hypothesis \ref{hypothesis2} and the Arzel\`a--Ascoli theorem.

\begin{lemma}
\label{regularity-lemma}
Solutions of \eqref{allen-cahn} satisfy the gradient estimate
\begin{equation}
\label{gradient-regularity}
\lvert \nabla u(x) \rvert \lesssim \e^{-\dist (x, \partial C_i)} \text{ for } x \in C_i.
\end{equation}
Moreover, a similar estimate holds for the potential $W(u)$, that is,
\begin{equation}
\label{potential-regularity}
\lvert W(u(x)) \rvert \lesssim \e^{-\dist (x, \partial C_i)} \text{ for } x \in C_i.
\end{equation}
\end{lemma}

\begin{proof}
Let $\Omega$ be an open and bounded subset of the region $C_i$. We have that $u$ is a classical solution of \eqref{allen-cahn}, so from the uniform bound \eqref{u-bound} it follows that $u \in H^1(\Omega) \cap L^\infty(\Omega)$, and also $W(u) \in H^1(\Omega) \cap L^\infty(\Omega)$ and $\nabla_{\! u} W(u) \in H^1(\Omega) \cap L^\infty(\Omega)$, due to the continuity of $W$ and $\nabla_{\! u} W(u)$. Then, from the mean value theorem in $\Omega$ and Hypothesis \ref{hypothesis1}, we have
\begin{equation}\label{mean-value-theorem}
|\nabla_{\! u} W(u)| = |\nabla_{\! u} W(u) - \nabla_{\! u} W(a_i)| \leq |\partial^2 W(\hat{u})| |u(x) - a_i| \lesssim \e^{-\dist (x,\partial C_i)}
\end{equation}
for some $\hat{u}$ in $\Omega$, and also,
\[ W(u) = W(u) - W(a_i) = |\nabla_{\! u} W(\hat{u})| |u(x) - a_i| \lesssim \e^{-\dist (x,\partial C_i)}.\]

Now let $\Omega^{\prime\prime} \subset \subset \Omega^\prime \subset \subset \Omega$ be open and bounded subsets of $\Omega$. From the definition of the $H^1$-norm and interior elliptic regularity, we have that 
\begin{equation}\label{h1-estimate}
\| u_{,i} \|_{L^2(\Omega^\prime)} \leq \| u - a_i \|_{H^1(\Omega^\prime)} \lesssim \| u - a_i \|_{L^2(\Omega)} + \| \nabla_{\! u} W(u) \|_{L^2(\Omega)} \lesssim \e^{-\dist (x,\partial C_i)},
\end{equation}
using \eqref{mean-value-theorem} and Hypothesis \ref{hypothesis1}. 

Differentiating \eqref{allen-cahn} gives $\Delta u_{,i} = \partial^2 W(u) \cdot u_{,i}$, and since from higher interior regularity we get $u \in H^{3}_{\text{loc}} (\Omega)$, we have the estimate
\[ \| u_{,i} \|_{H^2(\Omega^{\prime\prime})} \lesssim \| u_{,i} \|_{L^2(\Omega^\prime)} + \| \partial^2 W(u) \cdot u_{,i} \|_{L^2(\Omega^\prime)} \lesssim \e^{-\dist (x,\partial C_i)}, \]
using \eqref{h1-estimate} and the continuity of the Hessian $\partial^2 W(u)$. Finally, from the Sobolev imbedding of $H^2$ in $L^\infty$ in $\R^3$, we get that
\[ \| u_{,i} \|_{L^\infty(\Omega^{\prime\prime})} \lesssim \e^{-\dist (x,\partial C_i)},\] 
from which the statement of the lemma follows.
\end{proof}

\begin{lemma}
\label{gradient-limits}
For solutions of equation \eqref{allen-cahn}, the following pointwise limits hold:
\begin{align}
&\lim u_{,1} (x) = \dot{U}(x_1), \text{ as } x_2 \to +\infty,\ x_3 \to +\infty,\label{connection-limit}\\
&\lim u_{,i} (x) = 0, \text{ as } x_2 \to +\infty,\ x_3 \to +\infty, \text{ for } i=2, 3,\label{zero-limit}
\end{align}
where without loss of generality\footnote{This is due to the invariance of the Laplacian under rotations and the continuity of $W$ and $\nabla_{\! u} W$.} we considered a coordinate system such that $x_1$ is the distance to an interface $\Gamma$, with $U$ its corresponding connection.
\end{lemma}

\begin{proof}
We choose an interface and an appropriately rotated coordinate system such that the $x_1$ axis is normal to the interface and measures the signed distance from it, and let $\Omega$ be an interval of $x_1$ that is compact and symmetric with respect to the origin.

In order to establish the limit \eqref{connection-limit}, we consider test functions $\phi \in C^{\infty}_{\text{c}} (\Omega)$ and integrate by parts to get
\[ \int_{-\infty}^{+\infty} u_{,1} \phi(x_1) \dd x_1 = \int_{-\infty}^{+\infty} u \phi^\prime (x_1) \dd x_1, \]
and also
\[ \lim_{\substack{x_2 \to +\infty \\ x_3 \to +\infty}} \int_{-\infty}^{+\infty} u \phi^\prime (x_1) \dd x_1 = \int_{-\infty}^{+\infty} U(x_1) \phi^\prime (x_1) \dd x_1 = -\int_{-\infty}^{+\infty} \dot{U}(x_1) \phi (x_1) \dd x_1, \]
utilizing Hypothesis \ref{hypothesis2}. Combining the last two equations, we have that
\begin{equation}
\label{weak-l}
\lim_{\substack{x_2 \to +\infty \\ x_3 \to +\infty}} \int_{-\infty}^{+\infty} u_{,1} \phi(x_1) \dd x_1 = \int_{-\infty}^{+\infty} \dot{U}(x_1) \phi (x_1) \dd x_1.
\end{equation}
Now define the sequence 
\[ \{ u_{,1} \}_k := u_{,1} (x_1, x_2 + k, x_3 + k), \]
which is uniformly bounded from \eqref{grad-u-bound} and also equicontinuous in $\Omega$, since its derivative $\{u_{,11}\}_k$ is continuous and $\Omega$ is compact. Then, the Arzel\`a--Ascoli theorem and \eqref{weak-l} yield that there exists a  subsequence which converges uniformly to $\dot{U}(x_1)$. Assuming that another subsequence converges to a different limit, we reach a contradiction from the uniqueness of the weak limit in \eqref{weak-l}, thus concluding the argument.

For the limit \eqref{zero-limit}, consider the sequence 
\[ \{u\}_k := u (x_1, x_2 + k, x_3 + k), \]
which is again uniformly bounded and equicontinuous in $\Omega$ from \eqref{u-bound}, \eqref{grad-u-bound}. Then, the Arzel\`a--Ascoli theorem gives that there exists a subsequence $\{u\}_{k_l}$ which converges uniformly over compact subsets of $\Omega \times \R^2$ to a solution $v$ of \eqref{allen-cahn}. Using now Hypothesis \ref{hypothesis2}, we have that $v = U(x_1)$, and as a result, for $i=2,3$, applying Arzel\`a--Ascoli to the sequence of the derivative gives
\[ \lim_{\substack{x_2 \to +\infty \\ x_3 \to +\infty}} u_{,i} (x) = v_{,i} (x) = U_{,i} (x_1) = 0. \qedhere \]
\end{proof}

\section{Derivation of Young's law}
\label{section:derivation}

In this section we will prove the following theorem.
\begin{theorem}
For the contact angles at the spine of a triod of intersecting interfaces $\Gamma_{12}$, $\Gamma_{23}$, $\Gamma_{31}$, Young's law holds in the form of a balance of forces relation, that is,
\begin{equation}
\label{youngs-law-balance}
\sigma_{12} \nu_{12} + \sigma_{23} \nu_{23} + \sigma_{31} \nu_{31} = 0,
\end{equation}
where $\sigma_{ij}$ is the action of the connection $U_{ij}$ of each interface $\Gamma_{ij}$ and $\nu_{ij}$ the corresponding unit conormal, that is, a unit vector that is tangent to $\Gamma_{ij}$ and normal to the spine.
\end{theorem}

\begin{proof}
Since the solutions of \eqref{allen-cahn} which we consider are constructed as minimizers over balls (see \cite{alikakos-fusco-arma,alikakos-new-proof,fusco1} for the variational setup of the problem), we take a ball $B_R$ centered at $(0,0,2R)$ in order to apply the divergence theorem on it using \eqref{div-free}, that is, the fact that the stress tensor $T$ is divergence-free. This gives
\begin{equation}
\label{div-theorem}
0 = \frac{1}{R} \int_{B_R} \dv T \dd x = \frac{1}{R} \int_{\partial B_R} T \nu  \dd S,
\end{equation}
where $\nu$ is the outer unit normal to the boundary $\partial B_R$. In what follows we will study the limit 
\[ \lim_{R \to +\infty} \frac{1}{R} \int_{\partial B_R} T \nu  \dd S \]
in order to utilize the hypotheses on the solutions at infinity. Note that we chose the center of $B_R$ in such a way that for $(x_1,x_2,x_3) \in \partial B_R$, we have $x_3 \neq 0$ and $x_3 \to +\infty$, as $R \to +\infty$.

The complication in applying the divergence theorem in our problem is that the surface of integration intersects with the spine at two points, where we have no information on the behavior of solutions. In our setup these are the two poles of $\partial B_R$, at $(0,0,R)$ and $(0,0,3R)$. To circumvent this, we perform a surgery by choosing two appropriately sized spherical caps around the north and south poles of the sphere. To this end, let $\psi_2 (R)$ be a small polar angle that defines the spherical caps $\mathcal{C}$ (see Figure \ref{fig}), where by $\mathcal{C}$ we denote the union of both caps at the two poles. We require that there holds
\begin{equation}
\label{psi2-condition1}
R \sin \psi_2(R) \to +\infty, \text{ as } R \to +\infty,
\end{equation}
such that the distance of the boundary of the caps to the spine grows as $R \to +\infty$, which also yields that the geodesic radius $R\psi_2 (R)$ of the caps grows as $R \to +\infty$. Moreover, we require that
\begin{equation}
\label{psi2-condition2}
R  \psi_2(R)^2 \to 0, \text{ as } R \to +\infty,
\end{equation}
such that the renormalized area of the caps (in the sense below) shrinks as $R \to +\infty$. To see this, first note that condition \eqref{psi2-condition2} also yields that
\begin{equation}
\label{psi2-condition3}
\psi_2(R) \to 0, \text{ as } R \to +\infty.
\end{equation}
The renormalized area of the caps can be easily calculated as a surface integral using spherical coordinates, that is,
\[ \frac{1}{R} \int_{\mathcal{C}}  \dd S = \frac{2}{R} \int_{0}^{2\pi}\!\!\! \int_{0}^{\psi_2} R^2 \sin\theta_2 \dd \theta_2 \dd \theta_1 = 4\pi R (1-\cos \psi_2). \]
Using \eqref{psi2-condition3} we have that $(1-\cos \psi_2(R)) = O(\psi_{2}(R)^{2})$, which via \eqref{psi2-condition2} gives
\begin{equation}
\label{lim-caps}
\lim_{R \to +\infty} \frac{1}{R} \int_{\mathcal{C}}  \dd S = 0.
\end{equation}
To sum up, we choose the size of the caps to be small enough so as not to matter in the integration, but at the same time large enough so that we always stay away from the singularity.
Then, for the integral of $T\nu$ on the caps $\mathcal{C}$, we have the estimate
\[ \bigg\lvert \frac{1}{R} \int_{\mathcal{C}} T \nu \dd S \bigg\rvert \leq \frac{1}{R} \int_{\mathcal{C}} \lvert T \rvert \lvert \nu \rvert  \dd S \lesssim \frac{1}{R} \int_{\mathcal{C}}  \dd S, \]
where we used the bounds \eqref{u-bound}, \eqref{grad-u-bound}, and estimate \eqref{potential-regularity} of Lemma \ref{regularity-lemma} for bounding the Frobenius norm $\lvert T \rvert$ by a constant. Using now \eqref{lim-caps}, we finally have
\[ \lim_{R \to +\infty} \frac{1}{R} \int_{\mathcal{C}} T \nu \dd S = 0. \]

\begin{figure}
\begin{center}
\includegraphics{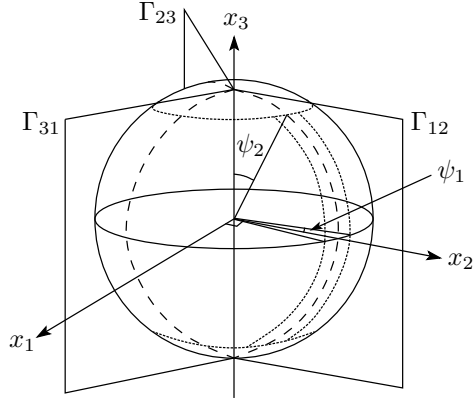}
\end{center}
\begin{picture}(0,0)
\put(-87,38){$x_{1}$}
\put(79,69){$x_{2}$}
\put(-6,160){$x_{3}$}
\put(65,120){$\Gamma_{12}$}
\put(-38,162){$\Gamma_{23}$}
\put(-82,120){$\Gamma_{31}$}
\put(75,102){$\psi_{1}$}
\put(-1,112){$\psi_{2}$}
\end{picture}
\caption{The sphere $\partial B_R$ centered at the spine, with two caps of polar angle $\psi_2(R)$ and a strip at geodesic distance $R\psi_1(R)$ around the intersection of the interface $\Gamma_{12}$ with $\partial B_R$.}
\label{fig}
\end{figure}

For the remaining part of the sphere, we will work separately for each interface that intersects it. For the interface $\Gamma_{12}$, which lies at azimuthal angle $\theta_1 = \frac{\pi}{2}$, we work with the slice 
\[ \mathcal{S} = \left\{ (\theta_1, \theta_2, r) ~\Big|~ \frac{\pi}{2} - \delta \leq \theta_1 \leq \frac{\pi}{2} + \delta \right\} \]
for a fixed angle $\delta$, such that no other interface intersects with the slice. To study the limit 
\[ \lim_{R \to +\infty} \frac{1}{R} \int_{\mathcal{S} \setminus (\mathcal{S}\cap\mathcal{C})} T\nu \dd S, \]
we distinguish two parts in $\mathcal{S} \setminus (\mathcal{S}\cap\mathcal{C})$, a neighborhood around the meridian at the intersection of the interface with the sphere and the rest. We take the set $\mathcal{N} \subset \mathcal{S} \setminus (\mathcal{S}\cap\mathcal{C})$, such that $\mathcal{N}$ is the strip that is contained between two planes parallel to $\Gamma_{12}$, one in the region $C_1$ and one in $C_2$, and at equal distance $R \sin \psi_1(R)$ from it (see Figure \ref{fig}). For the angle $\psi_1(R)$, we require that it satisfies
\begin{equation}
\label{psi1-condition1}
\psi_1(R) < \psi_2(R), \text{ with } \sqrt{2} \sin \psi_1(R) < \sin \psi_2(R),
\end{equation}
such that the azimuthal angle $\psi_1$ that defines the width of the strip $\mathcal{N}$ at the equator of $\partial B_R$ is strictly smaller than the polar angle $\psi_2$ that defines the caps. This condition forces $\mathcal{N}$ to be a subset of $\mathcal{S} \setminus (\mathcal{S}\cap\mathcal{C})$. Moreover, we require that
\begin{equation}
\label{psi1-condition2}
R \sin \psi_1(R) \to +\infty, \text{ as } R \to +\infty,
\end{equation}
such that the distance of the interface to the two planes that define $\mathcal{N}$ grows as $R\to+\infty$. From conditions \eqref{psi1-condition1} and \eqref{psi2-condition3} we also have that 
\begin{equation}
\label{psi1-condition3}
\psi_1(R) \to 0, \text{ as } R \to +\infty.
\end{equation}
An example of angles $\psi_1$, $\psi_2$ that satisfy all the above requirements is
\[ \psi_1(R) = R^{-4/5} \quad \text{and} \quad \psi_2(R) = R^{-3/4} \]
with condition \eqref{psi1-condition1} holding true for $R>1025$ for this particular choice.

Given the following decomposition of the set $\mathcal{S} \setminus (\mathcal{S}\cap\mathcal{C})$,
\[ \mathcal{S} \setminus (\mathcal{S}\cap\mathcal{C}) = \mathcal{N} ~\cup~ \left( (\mathcal{S} \setminus (\mathcal{S}\cap\mathcal{C})) \setminus \mathcal{N} \right), \]
we have the estimate
\begin{equation}
\label{estimate-on-the-rest}
\begin{split}
\bigg\lvert \frac{1}{R} \int_{(\mathcal{S} \setminus (\mathcal{S}\cap\mathcal{C})) \setminus \mathcal{N}} T \nu \dd S \bigg\rvert 
&\leq \frac{1}{R} \int_{(\mathcal{S} \setminus (\mathcal{S}\cap\mathcal{C})) \setminus \mathcal{N}} \lvert T \rvert \lvert \nu \rvert \dd S \\ 
&\lesssim \frac{1}{R} \int_{(\mathcal{S} \setminus (\mathcal{S}\cap\mathcal{C})) \setminus \mathcal{N}} \e^{-\dist (x, \Gamma_{12})} \dd S \\
&\lesssim R \e^{-R\sin\psi_1(R)},
\end{split}
\end{equation}
using estimates \eqref{gradient-regularity}, \eqref{potential-regularity} from Lemma \ref{regularity-lemma} for estimating $\lvert T \rvert$ by the exponential and since the domain of integration is of order $O(R^2)$. Finally, taking the limit as $R \to +\infty$ and using condition \eqref{psi1-condition2}, we have that
\[ \lim_{R \to +\infty} \frac{1}{R} \int_{(\mathcal{S} \setminus (\mathcal{S}\cap\mathcal{C})) \setminus \mathcal{N}} T \nu \dd S = 0. \]

We turn now to the last part, which is the integral on the strip $\mathcal{N}$. We parametrize $\mathcal{N}$ as the graph of
\[ f_R (x_1, x_3) = \sqrt{R^2 - x_{1}^{2} - (x_3 - 2R)^2} \]
for $x_1 \in (-R\sin\psi_1, R\sin\psi_1)$, $x_2 \in (R\sqrt{\sin^2 \psi_2 - \sin^2 \psi_1},R)$, $x_3 \in (2R - R\cos\psi_2,$ $2R + R\cos\psi_2)$, where $x_1 = x_2 = x_3 = 0$ is the origin, which is not the center of the sphere. We set $y_1 = x_1$, $y_2 = x_2$, and $y_3 = x_3 - 2R$, such that $(y_1,y_2,y_3) \in \partial B_R$ with $y_{1}^{2} + y_{2}^{2} + y_{3}^{2} = R^2$. Then, $\mathcal{N}$ is the graph of 
\[ f_R (y_1,y_3) = \sqrt{R^2 - y_{1}^{2} - y_{3}^{2}} \]
for
\begin{equation}
\label{intervals}
\begin{cases} 
y_1 \in (-R\sin\psi_1, R\sin\psi_1),\smallskip\\ 
y_2 \in (R\sqrt{\sin^2 \psi_2 - \sin^2 \psi_1},R),\smallskip\\
y_3 \in (-R\cos\psi_2, R\cos\psi_2).
\end{cases}
\end{equation}
For the surface element we calculate
\[ \dd S = \sqrt{1+ \left( \frac{\partial f_R}{\partial y_1} \right)^2 + \left( \frac{\partial f_R}{\partial y_3} \right)^2} \dd y_3 \dd y_1= \frac{R}{y_2} \dd y_3 \dd y_1,\] 
where $y_2 = \sqrt{R^2 - y_{1}^{2} - y_{3}^{2}}$, while the outer unit normal is
\[ \nu = \frac{y}{R}. \]
Using this parametrization, the integral on $\mathcal{N}$ is written as 
\begin{equation}
\label{parametrized} 
\frac{1}{R} \int_\mathcal{N} T\nu \dd S = \frac{1}{R} \int_{-R\sin\psi_1}^{R\sin\psi_1} \int_{-R\cos\psi_2}^{R\cos\psi_2} T(v) \frac{y}{y_2} \dd y_3 \dd y_1,
\end{equation}
where
\[ u(x_1, x_2, x_3) = u(y_1, y_2, y_3 + 2R) =: v(y_1, y_2, y_3), \]
and $v_{,i} (y) = u_{,i} (x)$ for $i=1,2,3$, so $T(v(y)) = T(u(x))$.

To take the limit as $R \to +\infty$ in \eqref{parametrized} and apply Hypotheses \ref{hypothesis1} and \ref{hypothesis2}, we use Lebesgue's dominated convergence theorem. The components of the vector quantity to be integrated on the right-hand side of \eqref{parametrized} are given by
\begin{equation}
\label{ty-components}
\left(T(v) \frac{y}{y_2} \right)_{\! i} = T_{ij}(v) \frac{y_j}{y_2} \text{ for } i=1,2,3,
\end{equation}
using the summation convention. To check whether dominated convergence applies for each component, we write the corresponding integral as
\[ \int_{-\infty}^{+\infty} \bigg( \frac{1}{R} \chi_{[-R\sin\psi_1,R\sin\psi_1]} \int_{-R\cos\psi_2}^{R\cos\psi_2} T_{ij}(v) \frac{y_j}{y_2} \dd y_3 \bigg) \dd y_1, \]
where $\chi$ is the characteristic function, and we would like to show that the quantity in parentheses is dominated by some integrable function of $y_1$. Using the estimates \eqref{gradient-regularity} and \eqref{potential-regularity} of Lemma \ref{regularity-lemma}, in $\mathcal{N}$ there holds $\lvert T_{ij}(v) \rvert \lesssim \e^{-\lvert y_1 \rvert}$, which gives the estimate
\begin{equation*}
\begin{split}
\bigg\lvert \frac{1}{R} \chi_{[-R\sin\psi_1,R\sin\psi_1]} \int_{-R\cos\psi_2}^{R\cos\psi_2} T_{ij}(v) \frac{y_j}{y_2} \dd y_3 \bigg\rvert
&\leq \frac{1}{R} \int_{-R\cos\psi_2}^{R\cos\psi_2} \lvert T_{ij}(v) \rvert \frac{\lvert y_j \rvert}{y_2} \dd y_3 \\
&\lesssim \frac{1}{R} \int_{-R\cos\psi_2}^{R\cos\psi_2} \e^{-\lvert y_1 \rvert} \frac{\lvert y_j \rvert}{y_2} \dd y_3 \\
&= \e^{-\lvert y_1 \rvert} \bigg( \frac{1}{R}  \int_{-R\cos\psi_2}^{R\cos\psi_2} \frac{\lvert y_j \rvert}{y_2} \dd y_3 \bigg).
\end{split}
\end{equation*}
Setting
\[ I_j = \frac{1}{R}  \int_{-R\cos\psi_2}^{R\cos\psi_2} \frac{\lvert y_j \rvert}{y_2} \dd y_3 \text{ for } j=1,2,3, \]
we argue that $I_j \leq 2$, for $j=1,2,3$, and, as a consequence, dominated convergence applies in the limit $R\to +\infty$.

For $j=1$, using the extremum values of $y_1$, $y_2$ in the intervals \eqref{intervals} and condition \eqref{psi1-condition1}, we have
\[
\begin{split}
I_1 &= \frac{1}{R}  \int_{-R\cos\psi_2}^{R\cos\psi_2} \frac{\lvert y_1 \rvert}{y_2} \dd y_3 
\leq \frac{1}{R}  \int_{-R\cos\psi_2}^{R\cos\psi_2} \frac{R \sin \psi_1}{R \sqrt{\sin^2 \psi_2 - \sin^2 \psi_1}} \dd y_3\\
& < \frac{1}{R}  \int_{-R\cos\psi_2}^{R\cos\psi_2} \frac{\sin \psi_1}{\sin \psi_1} \dd y_3 = \frac{1}{R}  \int_{-R\cos\psi_2}^{R\cos\psi_2} \dd y_3 = 2 \cos \psi_2 \leq 2,
\end{split}
\]
since $\sin \psi_1 < \sqrt{\sin^2 \psi_2 - \sin^2 \psi_1}$ from \eqref{psi1-condition1}.

For $j=2$, we have
\[
I_2 = \frac{1}{R}  \int_{-R\cos\psi_2}^{R\cos\psi_2} \frac{\lvert y_2 \rvert}{y_2} \dd y_3 
= \frac{1}{R}  \int_{-R\cos\psi_2}^{R\cos\psi_2} \dd y_3 = 2 \cos \psi_2 \leq 2,
\]
since $y_2 > 0$ in $\mathcal{N}$.

Finally, for $j=3$, we change variables to
\[ \tilde{y}_3 = \frac{1}{R} y_3 \text{ with } \!\dd \tilde{y}_3 = \frac{1}{R} \dd y_3, \]
and, using the extremum values of $y_1$ from \eqref{intervals}, we estimate
\[
\begin{split}
I_3 &= \frac{1}{R}  \int_{-R\cos\psi_2}^{R\cos\psi_2} \frac{\lvert y_3 \rvert}{y_2} \dd y_3 
= \int_{-\cos\psi_2}^{\cos\psi_2} \frac{R\lvert \tilde{y}_3 \rvert}{y_2} \dd \tilde{y}_3
= \int_{-\cos\psi_2}^{\cos\psi_2} \frac{R\lvert \tilde{y}_3 \rvert}{\sqrt{R^2 - y_{1}^{2} - R^2 \tilde{y}_{3}^{2}}} \dd \tilde{y}_3\\
&\leq \int_{-\cos\psi_2}^{\cos\psi_2} \frac{\lvert \tilde{y}_3 \rvert}{\sqrt{1 - \sin^2 \psi_1 - \tilde{y}_{3}^{2}}} \dd \tilde{y}_3 
= 2 \int_{0}^{\cos\psi_2} \frac{\tilde{y}_3}{\sqrt{1 - \sin^2 \psi_1 - \tilde{y}_{3}^{2}}} \dd \tilde{y}_3,
\end{split}
\]
since the function $\lvert \tilde{y}_3 \rvert \big/ \sqrt{1 - \sin^2 \psi_1 - \tilde{y}_{3}^{2}}$ is even. We explicitly calculate the last integral to get
\begin{align*}
2 &\int_{0}^{\cos\psi_2} \frac{\tilde{y}_3}{\sqrt{1 - \sin^2 \psi_1 - \tilde{y}_{3}^{2}}} \dd \tilde{y}_3
= -2 \int_{0}^{\cos\psi_2} \left( \sqrt{1 - \sin^2 \psi_1 - \tilde{y}_{3}^{2}}\, \right)^\prime \dd \tilde{y}_3 \\
&\quad = -2 \sqrt{1 - \sin^2 \psi_1 - \cos^2 \psi_2} + 2 \sqrt{1 - \sin^2 \psi_1} \\ 
&\quad = - 2 \sqrt{\sin^2 \psi_2 - \sin^2 \psi_1}  + 2 \cos \psi_1 \leq 2 \cos \psi_1 \leq 2.
\end{align*}

To conclude with the calculation of the limit as $R \to +\infty$ in \eqref{parametrized}, we distinguish the following limits in $\mathcal{N}$, as consequences of Hypothesis \ref{hypothesis2} and Lemma \ref{gradient-limits}:
\begin{equation}
\label{v-limits}
\begin{cases}
\lim_{R\to+\infty} v(y) = \lim_{\substack{x_2\to+\infty\\ x_3 \to+\infty}} u(x) = U_{12} (x_1),\\
\lim_{R\to+\infty} v_{,1}(y) = \lim_{\substack{x_2\to+\infty\\ x_3 \to+\infty}} u_{,1}(x) = \dot{U}_{12} (x_1),\\
\lim_{R\to+\infty} v_{,2}(y) = \lim_{\substack{x_2\to+\infty\\ x_3 \to+\infty}} u_{,2}(x) = 0,\\
\lim_{R\to+\infty} v_{,3}(y) = \lim_{\substack{x_2\to+\infty\\ x_3 \to+\infty}} u_{,3}(x) = 0.
\end{cases}
\end{equation}
Using the extremum values of the intervals in \eqref{intervals} and condition \eqref{psi1-condition1}, we also have that in $\mathcal{N}$ there holds
\begin{equation}
\label{y1y2-limit}
\frac{1}{R} \left\lvert \frac{y_1}{y_2} \right\rvert \leq \frac{1}{R} \frac{R\sin\psi_1}{R\sqrt{\sin^2\psi_2 - \sin^2 \psi_1}} < \frac{1}{R} \frac{\sin\psi_1}{\sin\psi_1} \to 0, \text{ as } R \to +\infty,
\end{equation}
\begin{equation}
\label{y3y2-limit}
\frac{1}{R} \left\lvert \frac{y_3}{y_2} \right\rvert \leq \frac{1}{R} \frac{R\cos\psi_2}{R\sqrt{\sin^2\psi_2 - \sin^2 \psi_1}} < \frac{1}{R} \frac{\cos\psi_2}{\sin\psi_1} \to 0, \text{ as } R \to +\infty,
\end{equation}
where in the last limit we also used condition \eqref{psi1-condition2} for the limit of the denominator. Using now \eqref{y1y2-limit}, \eqref{y3y2-limit} and since the elements of $T$ are bounded by a constant from \eqref{u-bound}, \eqref{grad-u-bound}, and estimate \eqref{potential-regularity} of Lemma \ref{regularity-lemma}, we have
\[\lim_{R\to+\infty} \frac{1}{R} T(v) \frac{y}{y_2} = \lim_{R\to+\infty} \frac{1}{R} (T_{12}, T_{22}, 
T_{32})^\top, \]
that is, only the components for $j=2$ in \eqref{ty-components} do not vanish in the limit.
But, using \eqref{v-limits}, we further have that
\[ \lim_{R\to+\infty} \frac{1}{R} T_{12} = \lim_{R\to+\infty} \frac{1}{R} v_{,1} \!\cdot v_{,2} = 0\quad\text{and}\quad \lim_{R\to+\infty} \frac{1}{R} T_{32} = \lim_{R\to+\infty} \frac{1}{R} v_{,3} \!\cdot v_{,2} = 0. \]

Finally,
\[ \lim_{R\to+\infty} \frac{1}{R} \int_\mathcal{N} T\nu \dd S = 
\bigg( \lim_{R\to+\infty} \frac{1}{R} \int_{-R\sin\psi_1}^{R\sin\psi_1} \int_{-R\cos\psi_2}^{R\cos\psi_2} T_{22}(v) \dd y_3 \dd y_1 \bigg) (0,1,0)^\top. \]
Plugging in the component $T_{22}$ into the last integral, we calculate the limit
\[ 
\lim_{R\to+\infty} \frac{1}{R} \int_{-R\sin\psi_1}^{R\sin\psi_1} \int_{-R\cos\psi_2}^{R\cos\psi_2} \frac{1}{2} \Big( \lvert v_{,2} \rvert^2 - \lvert v_{,1} \rvert^2 - \lvert v_{,3} \rvert^2 - 2 W(v) \Big) \dd y_3 \dd y_1
\]
via the change of variables $y_3 = R \tilde{y}_3$, which gives 
\[
\lim_{R\to+\infty} \int_{-R\sin\psi_1}^{R\sin\psi_1} \int_{-\cos\psi_2}^{\cos\psi_2} \frac{1}{2} \Big( \lvert v_{,2} \rvert^2 - \lvert v_{,1} \rvert^2 - \frac{1}{R^2} \lvert v_{,3} \rvert^2 - 2 W(v) \Big) \dd \tilde{y}_3 \dd y_1
\]
with a slight abuse of notation for $v_{,3}$. Passing the limit inside the last integral and using conditions \eqref{psi2-condition3} and \eqref{psi1-condition2}, the limits in \eqref{v-limits} give
\begin{align*}
&\int_{-\infty}^{+\infty} \int_{-1}^{1} -\left( \frac{1}{2} \lvert \dot{U}_{12}(y_1) \rvert^2 + W(U_{12}(y_1)) \right) \dd \tilde{y}_3 \dd y_1\\
&\quad=-2\int_{-\infty}^{+\infty} \left( \frac{1}{2} \lvert \dot{U}_{12}(y_1) \rvert^2 + W(U_{12}(y_1)) \right) \dd y_1\\
&\quad=-2\sigma_{12}.
\end{align*}
Thus, we have shown that for the slice $\mathcal{S}$ around the interface $\Gamma_{12}$ there holds
\[ \lim_{R\to+\infty} \frac{1}{R} \int_{\mathcal{S}} T \nu \dd S = -2\sigma_{12}\nu_{12},\]
where $\nu_{12} = (0,1,0)^\top$. 

Since the stress tensor $T$ is invariant under rotations, we can apply the same procedure for the other two interfaces for appropriately rotated coordinate systems and appropriate slices (in order to cover the whole sphere) to get 
\[ \sigma_{12}\nu_{12} + \sigma_{23}\nu_{23} + \sigma_{31}\nu_{31} =0,\]
using \eqref{div-theorem}, where the $\nu_{ij}$'s are the conormals of the corresponding interfaces $\Gamma_{ij}$. This concludes the proof.
\end{proof}

We remark that the balance of forces relation \eqref{youngs-law-balance} is equivalent to Young's law \eqref{youngs-law}. This can be easily deduced by multiplying \eqref{youngs-law-balance} with the unit normal of each interface.

\nocite{*}
\bibliographystyle{plain}

\end{document}